\documentclass[12pt, reqno]{amsart}
\usepackage{amsfonts, amsmath, amsthm, amscd, amsfonts, amssymb, graphicx, color}
\usepackage{tikz-cd}
\usepackage[bookmarksnumbered, colorlinks, plainpages]{hyperref}
\newtheorem{theorem}{Theorem}[section]
\newtheorem{lemma}[theorem]{Lemma}
\newtheorem{proposition}[theorem]{Proposition}

\newtheorem{corollary}[theorem]{Corollary}
\newtheorem{definition}[theorem]{Definition}

\usepackage[bookmarksnumbered, colorlinks, plainpages]{hyperref}
\hypersetup{colorlinks=true,linkcolor=blue, anchorcolor=green, citecolor=cyan, urlcolor=red, filecolor=magenta, pdftoolbar=true}

\usepackage{comment}
\excludecomment{mysection}

\begin{document}
	
\title[Ordered normed spaces of functions of bounded variation]{Ordered normed spaces of functions of bounded variation}
\author{Amit Kumar}
	
\address{Discipline of Mathematics, School of Basic Sciences, Indian Institute of Technology Bhubaneswar, Argul, Bhubaneswar, Pin - 752050, Odisha (State), India.}

\email{\textcolor[rgb]{0.00,0.00,0.84}{amit231291280895@gmail.com} and \textcolor[rgb]{0.00,0.00,0.84}{ak90@iitbbs.ac.in}}

	

\begin{abstract}
In this paper, we define and study the space of all the functions of bounded variation $f:[x,y]\to \mathbb{Y}$ denoted by $\mathcal{BV}[x,y],$ where $[x,y]$ is an ordered interval and $\mathbb{Y}$ is an absolute order unit space having vector lattice structure. By default, under the order structure of $\mathbb{Y},$ the space $\mathcal{BV}[x,y]$ forms a nearer absolute order unit space structure and in some cases it turns out to be an absolute order unit space (in fact, a unital $AM$-space). By help of variation function, we also define a different kind of order structure on the space $\mathcal{BV}[x,y]$ that also makes  $\mathcal{BV}[x,y]$ a nearer absolute order unit space structure. Later, we also show that under certain conditions this ordering induces a complete norm on $\mathcal{BV}[x,y].$ 
\end{abstract}

\thanks{The author was financially supported by the Institute Post-doctoral Fellowship of IIT Bhubaneswar, India.}

\maketitle

\section{Introduction}
The theory of functions of bounded variation is well known in Mathematical Analysis. Functions of bounded variation are also called $\mathcal{BV}$-functions. In Complex Analysis, $\mathcal{BV}$-functions are used to defined arc-length of smooth curves. In other words, if $\gamma:[0,1]\to \mathbb{C}$ is a continuously differentiable function, then $\gamma$ is a $\mathcal{BV}$-function and the total variation of $\gamma$ is given by $\mathcal{V}(\gamma)=\displaystyle\int_{0}^{1}\vert \gamma'(t)\vert dt.$ In 1881, Camille Jordan initated the theory of $\mathcal{BV}$-functions of a single variable to deal with the convergence in fourier series \cite{CJFS}. On the other hand, the theory of $\mathcal{BV}$-functions of several variables was initated by Leonida Tonelli in 1926 (see \cite{LCLD}). However, $\mathcal{BV}$-functions of several variables were formally defined and studied by Lamberto Cesari in 1936 \cite{LSSF}. The $\mathcal{BV}$-functions forms an algebra of discontinuous functions having first order derivative almost everywhere. This is a major importance of $\mathcal{BV}$-functions in Mathematics, Physics and Engineering as it helps to define a generalized solutions of of non-linear problems that involves functionals, ordinary and partial differential equations. It is worth to notice that triangle inequality plays a crucial role in the study of the $\mathcal{BV}$-functions. The triangle inequality holds in $\mathbb{R}$ and $\mathbb{C}$ that is why it is possible to study $\mathcal{BV}$-functions in these spaces. For more informations about $\mathcal{BV}$-functions, we refer to see \cite{JBCF, RLWA} and references therein.

Order structure is one of important parts of the $C^*$-algebras. It characterizes $C^*$-algebras. Its fundamental importances can be found in \cite{B06, RVK, Kad51, RJ83, Kak, GKP} and references therein. Parallel theory of order structure has also been developed in vector spaces, for details see \cite{MA71, CDOB, J72, HHS74, WN73}. Being inspired by the richness of order structure, Karn also started working on the order theoretic aspects of $C^*$-algebras. Some of his related works can be seen in \cite{K10, K14, K16, K18, K19}.

In \cite{K18}, Karn introduced and studied the notion of absolutely ordered spaces and absolute order unit spaces. Under the condition \cite[Theorem 4.12]{K16}(triangle inequality holds), absolutely ordered spaces turn out to be vector lattices and under the same condition absolute order unit spaces turn out to be unital $AM$-spaces. That was the reason, Karn named ``absolutely ordered spaces" as ``non-commutative vector lattice models". Therefore, it is obvious question for the study of $\mathcal{BV}$-functions in absolutely ordered spaces. In this paper, we have defined and studied the notion of $\mathcal{BV}$-functions in absolutely ordered spaces. Finally, our aim is to show that $\mathcal{BV}$-functions forms ordered normed spaces.

The development of the paper is as follows. In the second section, we recall the preliminaries which are essential to write this paper. In the third section, we define $\mathcal{BV}$-functions and study their basic properties (Theorems \ref{6} and \ref{5}, and Lemma \ref{4}). We investigate when the space of $\mathcal{BV}$-functions forms an absolute order unit space and $AM$-space (Theorem \ref{7}). In the fourth section (last section), we define variation function for $\mathcal{BV}$-functions and study its basic properties in terms of $\mathcal{BV}$-functions (Theorem \ref{13}). We also construct some norms on the space of $\mathcal{BV}$-functions under which it turn out to be ordered normed spaces (Theorem \ref{7}, Corollaries \ref{15} and \ref{16}). Under the order completeness, one of these norms turns out to be a complete norm (Theorem \ref{3}).  

\section{Preliminaries}
Let $\mathbb{X}$ be a real vector space. A non-empty subset $\mathbb{X}^+$ of $\mathbb{X}$ is said to be a cone, if $x+y$ and $\alpha x\in \mathbb{X}^+$ for all $x,y\in \mathbb{X}^+$ and $\alpha\in \mathbb{R}^+\cup \lbrace 0\rbrace.$ Then $(\mathbb{X}, \mathbb{X}^+)$ is said to be a \emph{real ordered vector space}. Given a partial ordered space $(\mathbb{X}, \leq),$ put $\mathbb{X}^+=\lbrace x \in \mathbb{X}:x\geq 0\rbrace.$ Then $x \leq y$ if $y - x \in \mathbb{X}^+.$ In this way, $\leq $ is unique with the following properties: $x \le x$ for all $x\in \mathbb{X},$  $x \leq z$ provided $ x \leq y$ and $y \leq z$ and, $x + z \leq y + z$ and $\alpha x \leq \alpha y$ provided $x \leq y$, $z \in \mathbb{X}$ and $\alpha\in \mathbb{R}^+.$ If $\mathbb{X}^+ \cap - \mathbb{X}^+ = \{ 0 \},$ then the cone $\mathbb{X}^+$ is called \emph{proper} and if $\mathbb{X} = \mathbb{X}^+ - \mathbb{X}^+,$ then it is called \emph{generating}. It is worth to note that $\mathbb{X}^+$ is proper if and only if $\leq$ is anti-symmetric. 

An element $e\in \mathbb{X}^+$ is called order unit for $\mathbb{X}$ provided for every $x\in \mathbb{X},$ we have $\epsilon e\pm x\in \mathbb{X}^+$ for some $\epsilon > 0.$ The cone $\mathbb{X}^+$ is called \emph{Archimedean} provided for $x\in \mathbb{X}$ and a fixed $y\in \mathbb{X}^+$ such that $\epsilon y+x\in \mathbb{X}^+$ for all $\epsilon > 0,$ it turns out that $x\in \mathbb{X}^+.$

In a real ordered vector space $(\mathbb{X}, \mathbb{X}^+)$ with order unit $e$ and such that $\mathbb{X}^+$ is proper and Archimedean, we can always define a norm on $\mathbb{X}$ in the following way: $$\Vert x\Vert := \inf \lbrace \epsilon > 0: \epsilon e \pm x \in \mathbb{X}^+ \rbrace.$$ This is called norm determined by $e.$ Moreover, $\mathbb{X}^+$ is norm-closed as well as $\Vert x\Vert e\pm x \in \mathbb{X}^+$ for every $x\in \mathbb{X}.$ In this case, $\mathbb{X}$ is called an \emph{order unit space} and we denote it by $(\mathbb{X}, e).$

Let $\mathbb{X}$ be a real ordered vector space and $\mathbb{S}$ be a non-empty subset of $\mathbb{X}.$ Then $\mathbb{S}$ is called bounded above in $\mathbb{X}$ if there exists $z\in \mathbb{X}$ such that $x\leq z$ for all $x\in \mathbb{S}.$ In this case, we say that $\mathbb{S}$ is bounded above by $z$ and $z$ is called upper bound of $\mathbb{S}.$ Similarly, $\mathbb{S}$ is called bounded below in $\mathbb{X}$ if there exists $w\in \mathbb{X}$ such that $w\leq x$ for all $x\in \mathbb{X}.$ In this case, we say that $\mathbb{S}$ is bounded below by $w$ and $w$ is called lower bound of $\mathbb{S}.$ We say that $z \in \mathbb{X}$ is supremum of $\mathbb{S}$ if $z$ is upper bound of $\mathbb{S}$ and whenever $w\in \mathbb{X}$ is any other upper bound of $\mathbb{S},$ it turns out that $z\leq w.$ In this case, we write: $\sup\lbrace x:x\in \mathbb{S}\rbrace=z.$ Similarly, we say that $w \in \mathbb{X}$ is infimum of $\mathbb{S}$ if $w$ is lower bound of $\mathbb{S}$ and whenever $z$ is any other lower bound of $\mathbb{S},$ it turns out that $z\leq w.$ In this case, we write: $\inf\lbrace x:x\in \mathbb{S}\rbrace=w.$ Note that $\sup\lbrace x:x\in \mathbb{S}\rbrace$ exists in $\mathbb{X}$ if and only if $\inf \lbrace -x:x\in \mathbb{S}\rbrace$ exists in $\mathbb{X}.$ In this case, $\sup\lbrace x:x\in \mathbb{S}\rbrace=-\inf\lbrace -x:x\in \mathbb{S}\rbrace.$

A real ordered vector space $\mathbb{X}$ is called \emph{vector lattice} provided  $\sup\lbrace x, y\rbrace$ exists in $\mathbb{X}$ for every pair $x$ and $y\in \mathbb{X}.$ In a vector lattice, we write: $x\vee y=\sup \lbrace x,y\rbrace,x\wedge y=\inf \lbrace x,y\rbrace$ and $\vert x\vert=x\vee (-x).$ 

A vector lattice $\mathbb{X}$ is called Dedekind complete if supremum of every non-empty bounded above subset of $\mathbb{X}$ exists in $\mathbb{X}.$

Let $(\mathbb{X},\mathbb{X}^+)$ be a vector lattice with a norm $\Vert \cdot \Vert$ such that $(\mathbb{X},\Vert \cdot \Vert)$ forms a Banach space. Then $(\mathbb{X},\mathbb{X}^+)$ called a \emph{$AM$-space} provided the following two conditions hold:
\begin{enumerate}
\item[(1)] $\vert x \vert \leq \vert y \vert$ implies $\Vert x \Vert \leq \Vert y\Vert$ for every pair $x,y \in \mathbb{X}.$
\item[(2)] For $x,y \in \mathbb{X}^+,$ we have $\Vert x \vee y\Vert = \textrm{max} \lbrace \Vert x\Vert, \Vert y\Vert \rbrace.$ 
\end{enumerate}

Let's recall the notion of absolutely ordered spaces introduced by Karn as a possible non-commutative model for vector lattices \cite{K18}.

\begin{definition} \cite[Definition 3.4]{K18}\label{30}
Let $(\mathbb{X}, \mathbb{X}^+)$ be a real ordered vector space and let $\vert\cdot\vert: \mathbb{X} \to \mathbb{X}^+$ be a mapping satisfying the following conditions:               
     \begin{enumerate}
          \item[(a)] $\vert x \vert = x$ if $x \in \mathbb{X}^+.$
          \item[(b)] $\vert x \vert \pm x \in \mathbb{X}^+$ for all $x \in \mathbb{X}.$
          \item[(c)] $\vert \alpha \cdot x \vert = \vert \alpha \vert \cdot \vert x \vert$ for all $x \in \mathbb{X}$ and $\alpha \in \mathbb{R}.$
          \item[(d)] If $x, y$ and $z \in \mathbb{X}$ with $\vert x - y \vert = x + y$ and $0 \leq z \leq y,$ then $\vert x - z \vert = x + z.$
          \item[(e)] If $x, y$ and $z \in \mathbb{X}$ with $\vert x - y \vert = x + y$ and $\vert x - z \vert = x + z,$ then $\vert x - \vert y \pm z \vert \vert = x + \vert y \pm z \vert.$ 
     \end{enumerate}   
     
Then $(\mathbb{X}, \mathbb{X}^+, \vert \cdot \vert)$ is said to be an \emph{absolutely ordered space}.   
\end{definition}  

The following result explains that absolutely ordered space is very near to a lattice structure that is why Karn called it possible non-commutative model for vector lattices.

\begin{theorem}\label{2}
Let $(\mathbb{X},\mathbb{X}^+,\vert \cdot\vert)$ be an absolutely ordered space. For $y,z\in \mathbb{X},$ put $$y \dot{\vee} z:=\frac{1}{2}(y+z+\vert y-z\vert).$$ Then the following statements are equivalent:
\begin{enumerate}
\item[(1)] $y \dot{\vee} z=\sup \lbrace y,z\rbrace$ for all $y,z\in \mathbb{X}.$
\item[(2)] $\dot{\vee}$ is associative in $\mathbb{X}.$
\item[(3)] $\pm y\leq x$ implies $\vert y\vert \leq x$ for all $x,y\in \mathbb{X}.$
\item[(4)] $\vert y+z\vert \leq \vert y\vert +\vert z\vert.$
\end{enumerate}
\end{theorem}
     
Next, we recall some variants of orthogonalities in absolutely ordered spaces.     
     
\begin{definition}[\cite{K18}, Definition 3.6]
	Let $(\mathbb{X}, \mathbb{X}^+, \vert\cdot\vert)$ be an absolutely ordered space and let $\| \cdot \|$ be a norm on $\mathbb{X}.$ 
	\begin{enumerate} 
		\item[(a)] For $x, y \in \mathbb{X}^+$, we say that $x$ is \emph{orthogonal} to $y$ ($x \perp y$) if, $\vert x - y \vert = x + y.$ Put $x^+ := \frac{1}{2}(\vert x \vert + x)$ and $x^- := \frac{1}{2}(\vert x \vert - x).$ In this case, $x = x^+ - x^-$ and $\vert x \vert = x^+ + x^-$ so that $x^+ \perp x^-.$ This decomposition turns out to be unique in the sense: $x = x_1 - x_2$ such that $x_1 \perp x_2$ implies $x_1 = x^+$ and $x_2 = x^-.$ Therefore each element in $\mathbb{X}$ owns a unique orthogonal decomposition in $\mathbb{X}^+.$
		\item[(b)] For $x, y \in \mathbb{X}^+$, we say that $x$ is \emph{$\infty$-orthogonal} to $y$ ($x \perp_\infty y$) if, $\Vert \alpha x + \beta y\Vert = \max \lbrace \Vert \alpha u \Vert, \Vert \beta v \Vert \rbrace$ for all $\alpha, \beta \in \mathbb{R}.$
		\item[(c)] For $x, y \in \mathbb{X}^+$, we say that $x$ is \emph{absolutely $\infty$-orthogonal} to $y$ ($x \perp_\infty^a y$) if, $x_1 \perp_\infty y_1$ whenever $0 \leq x_1 \leq x$ and $0 \leq y_1 \leq y.$
	\end{enumerate}
\end{definition}
 
Now, we recall absolute order unit spaces. 

\begin{definition}[\cite{K18}, Definition 3.8]
Let $(\mathbb{X}, \mathbb{X}^+, \vert \cdot \vert)$ be an absolutely ordered space and let $\Vert \cdot \Vert$ be an order unit norm on $\mathbb{X}$ determined by the order unit $e$ such that $\mathbb{X}^+$ is $\Vert \cdot \Vert$-closed. Then $(\mathbb{X}, \mathbb{X}^+, \vert \cdot \vert, e)$ is called an \emph{absolute order unit space} if $\perp = \perp^a_\infty$ on $\mathbb{X}^+.$
\end{definition}

Note that the self-adjoint part of a unital C$^*$-algebra is an absolute order unit space \cite[Remark 3.9(1)]{K18}. More generally, every unital $JB$-algebra is also an absolute order unit space.     
     
\section{Functions of bounded variation and their properties}

Let $\mathbb{X}$ be a real ordered vector space and $x,y \in \mathbb{X}$ such that $x\leq y,$ the ordered interval $[x,y]$ in $\mathbb{X}$ is defined by $[x,y]=\lbrace z \in \mathbb{X}: x\leq z \leq y\rbrace.$

A partition $\mathcal{P}$ of $[x,y]$ is collection of points in $[x,y]$ such that $\mathcal{P}=\lbrace x=x_0 < x_1 < x_2 < \cdots < x_{n_\mathcal{P} - 1} < x_{n_{\mathcal{P}}}=y\rbrace.$

Let $\mathbb{Y}$ be an absolutely ordered space and $f:[x,y]\to \mathbb{Y}$ be a function. Let $\mathcal{P}$ be a partion of $[x,y].$ We consider the following summation over $\mathcal{P}:$ $$\displaystyle \sum_{i=1}^{n_{\mathcal{P}}} \vert f(x_i)-f(x_{i-1})\vert.$$ We denote it by $\Sigma^{x,y}_\mathcal{P}[f].$ Most of the times, we denote $\Sigma^{x,y}_\mathcal{P}[f]$ by $\Sigma_\mathcal{P}[f]$ if there is no ambuigity.

\begin{proposition}\label{10}
Let $\mathbb{Y}$ be a vector lattice structure and $\mathcal{P}_1$ and $\mathcal{P}_2$ be partitions of $[x,y]$ such that $\mathcal{P}_1 \subset \mathcal{P}_2,$ then $\Sigma_{\mathcal{P}_1}[f]\leq \Sigma_{\mathcal{P}_2}[f].$ In particular, $\vert f(y)-f(x)\vert \leq \Sigma_\mathcal{P}[f]$ for every partition $\mathcal{P}$ of $[x,y].$ 

\begin{proof}
Let $\mathcal{P}_1=\lbrace x=x_0 < x_1 < x_2 < \cdots < x_{n-1} < x_n=y\rbrace.$ Without loss of generality, we assume that $\mathcal{P}_2$ contains exactly one more point than $P_1.$ In this case, we have $\mathcal{P}_2=\lbrace x=x_0 < x_1 < x_2 < \cdots < x_{i-1} < z < x_i < \cdots < x_{n-1} < x_n=y\rbrace.$ By Theorem \ref{2}, we get that

\begin{eqnarray*}
\vert f(x_i)-f(x_{i-1})\vert & \leq & \vert (f(x_i)-f(z))+(f(z)-f(x_{i-1})) \vert \\
&\leq & \vert f(x_i)-f(z) \vert + \vert f(z)-f(x_{i-1}) \vert
\end{eqnarray*}
 
so that $\Sigma_{\mathcal{P}_1}[f]\leq \Sigma_{\mathcal{P}_2}[f].$

\end{proof}

\end{proposition}

Now, we introduce the notion of functions of bounded variation in absolutely ordered spaces.

\begin{definition}
Let $\mathbb{Y}$ be an absolutely ordered space and $f:[x,y] \to \mathbb{Y}$ be a function. Then $f$ is said to be of bounded variation, if \begin{center} $\sup \left \lbrace \sum_{\mathcal{P}}[f] :\textrm{$\mathcal{P}$ is a partition of $[x,y]$} \right \rbrace$ \end{center} exists in $\mathbb{Y}.$ If $f$ is of bounded variation, we say \begin{center} $\mathcal{V}(f,x,y)=\sup \left \lbrace \sum_{\mathcal{P}}[f]:\textrm{$\mathcal{P}$ is a partition of $[x,y]$} \right \rbrace$\end{center} the total variation of $f$ and we also write: $$\mathcal{BV}[x,y]=\lbrace f:[x,y]\to \mathbb{Y}~\textrm{is a function of bounded variation}~\rbrace.$$ Most of the times, we denote $\mathcal{V}(f,x,y)$ and $\mathcal{BV}[x,y]$ by $\mathcal{V}(f)$ and $\mathcal{BV}$ if there is no ambiguity. 
\end{definition}

The following result is immediate from Proposition \ref{10}.

\begin{corollary}\label{8}
Let $f\in \mathcal{BV}[x,y].$ Then $\vert f(y)-f(x)\vert \leq \mathcal{V}(f).$
\end{corollary}

Next, we study algebra of functions of bounded variations.

\begin{theorem}\label{6}
Let $\mathbb{Y}$ be a vector lattice and $f,g\in \mathcal{BV}[x,y].$ Then 
\begin{enumerate}
\item[(1)] $f$ is bounded.
\item[(2)] $\alpha f$ is also of bounded variation with $\mathcal{V}(\alpha f)=\vert \alpha \vert \mathcal{V}(f)$ for any $\alpha \in \mathbb{R}$. In particular, $-f$ is also of bounded variation with $\mathcal{V}(-f)=\mathcal{V}(f).$
\end{enumerate}
Moreover, if $\mathbb{Y}$ is dedekind complete, then 
\begin{enumerate}
\item[(3)]$f\pm g$ is also of bounded variation with $\mathcal{V}(f \pm g) \leq \mathcal{V}(f)+\mathcal{V}(g).$
\item[(4)]$\vert f\vert$ is also of bounded variation and $\mathcal{V}(\vert f\vert)\leq \mathcal{V}(f).$
\end{enumerate}
\end{theorem}

\begin{proof}
\begin{enumerate}
\item[(1)] Let $f\in \mathcal{BV}[x,y].$ For $z \in [x,y],$ we have  
\begin{eqnarray*}
\vert f(z) \vert &=& \vert (f(z)-f(x))+f(x))\vert \\
&\leq & \vert f(z)-f(x) \vert + \vert f(x))\vert \\
&\leq & \vert f(y)-f(z) \vert + \vert f(z)-f(x) \vert + \vert f(x))\vert \\
&\leq & \mathcal{V}(f)+\vert f(x) \vert.
\end{eqnarray*}

Hence $f$ is bounded as $\mathcal{V}(f)+\vert f(x) \vert$ is a fix element in $\mathbb{Y}.$ 
\item[(2)] For any $\alpha \in \mathbb{R}$ and partion $\mathcal{P}=\lbrace x=x_0 < x_1 < x_2 < \cdots < x_{n_{\mathcal{P}}-1} < x_{n_\mathcal{P}}=y\rbrace,$ we have $$\displaystyle \sum_{i=1}^{n_{\mathcal{P}}} \vert \alpha f(x_i)- \alpha f(x_{i-1})\vert = \vert \alpha \vert \displaystyle \sum_{i=1}^{n_{\mathcal{P}}} \vert f(x_i)-f(x_{i-1})\vert.$$ Thus $\alpha f$ is also of bounded variation with $\mathcal{V}(\alpha f)=\vert \alpha \vert \mathcal{V}(f).$ 
\item[(3)] Let $g\in \mathcal{BV}[x,y].$ By Theorem \ref{2}, we get that
\begin{eqnarray*}
\displaystyle \sum_{i=1}^{n_\mathcal{P}} \vert (f+g)(x_i)-(f+g)(x_{i-1})\vert &=& \displaystyle \sum_{i=1}^{n_{\mathcal{P}}} \vert (f(x_i)-f(x_{i-1}))+(g(x_i)-g(x_{i-1}))\vert \\
&\leq & \displaystyle \sum_{i=1}^{n_{\mathcal{P}}} (\vert f(x_i)-f(x_{i-1}) \vert + \vert g(x_i)-g(x_{i-1})\vert )\\
& = & \displaystyle \sum_{i=1}^{n_{\mathcal{P}}} \vert f(x_i)-f(x_{i-1}) \vert + \displaystyle \sum_{i=1}^{n_{\mathcal{P}}} \vert g(x_i)-g(x_{i-1})\vert \\
& \leq & \mathcal{V}(f) + \mathcal{V}(g). 
\end{eqnarray*}

Thus $f+g$ is of bounded variation with $\mathcal{V}(f+g) \leq \mathcal{V}(f) + \mathcal{V}(g).$ 

Next, $g$ is of bounded variation, by (2), we get that $-g$ is also of bounded variation with $\mathcal{V}(-g)=\mathcal{V}(g).$ Since $f$ and $-g$ are of bounded variation, we get that $f-g=f+(-g)$ is also of bounded variation with $\mathcal{V}(f-g) \leq \mathcal{V}(f)+\mathcal{V}(-g)=\mathcal{V}(f)+\mathcal{V}(g).$
\item[(4)] For any $i,$ we have 

\begin{eqnarray*}
\vert f(x_i)\vert &=& \vert f(x_i)-f(x_{i-1})+f(x_{i-1})\vert \\
&\leq & \vert f(x_i)-f(x_{i-1}\vert + \vert f(x_{i-1})\vert
\end{eqnarray*}

so that $$\vert f(x_i)\vert - \vert f(x_{i-1})\vert \leq \vert f(x_i)-f(x_{i-1})\vert.$$ Interchanging $x_i$ by $x_{i-1},$ we also get that $\vert f(x_{i-1})\vert - \vert f(x_i)\vert \leq \vert f(x_{i-1})-f(x_i)\vert.$ Finally, we get that $\pm (\vert f(x_i)\vert - \vert f(x_{i-1})\vert) \leq \vert f(x_i)-f(x_{i-1})\vert.$ By Theorem \ref{2}, we have $\vert \vert f(x_i)\vert - \vert f(x_{i-1})\vert \vert \leq \vert f(x_i)-f(x_{i-1})\vert.$

Since $f$ is a function of bounded variation and $\displaystyle \sum_{i=1}^{n_{\mathcal{P}}} \vert \vert f(x_i)\vert - \vert f(x_{i-1})\vert \vert \leq \displaystyle \sum_{i=1}^{n_{\mathcal{P}}} \vert f(x_i)-f(x_{i-1})\vert$ for any partition $\mathcal{P}$ of $[x,y],$ we conclude that $\vert f \vert$ is also of bounded variation and $\mathcal{V}(\vert f\vert)\leq \mathcal{V}(f).$
\end{enumerate}
\end{proof}

The following result tells that every monotone function turns out to be a function of bounded variation.

\begin{proposition}\label{1}
Let $\mathbb{Y}$ be an absolutely ordered space and $f:[x,y] \to \mathbb{Y}$ is monotone. Then $f\in \mathcal{BV}$ with $\mathcal{V}(f)=\vert f(y)-f(x) \vert.$

\begin{proof}
Let $f:[x,y] \to \mathbb{Y}$ be monotonically increasing and let $P=\lbrace x=x_0 < x_1 < x_2 < \cdots < x_{n-1} < x_{n_\mathcal{P}}=y\rbrace$ be a partition of $[x,y].$ Then $$f(x_i)-f(x_{i-1}) \geq 0~\textrm{for all}~i$$ so that 

\begin{eqnarray*}
\displaystyle \sum_{i=1}^{n_\mathcal{P}} \vert f(x_i)-f(x_{i-1})\vert &=& \displaystyle \sum_{i=1}^{n_\mathcal{P}} (f(x_i)-f(x_{i-1})) \\
&=& f(x_{n_\mathcal{P}})-f(x_0) \\
&=& f(y)-f(x).
\end{eqnarray*}

For any partition $P,$ we get that $\sum_\mathcal{P}[f]= f(y)-f(x).$ Hence $f\in \mathcal{BV}$ and $\mathcal{V}(f)=f(y)-f(x).$ 

Next, if $f$ is monotonically decreasing, then $-f$ is monotonically increasing and so $-f\in \mathcal{BV}[x,y]$ with $\mathcal{V}(-f)=f(x)-f(y).$ Consequently, by Theorem \ref{6}(2), $f\in \mathcal{BV}[x,y]$ with $\mathcal{V}(f)=\mathcal{V}(-f)=f(x)-f(y)=-(f(y)-f(x)).$ Finally, we conclude that every monotone function $f$ is of bounded variation with $\mathcal{V}(f)=\vert f(y)-f(x)\vert.$ 
\end{proof}

\end{proposition}

The notion of $\vert \cdot \vert$-prserving maps between absolutely ordered spaces has been introduced and studied by Karn and the author in \cite{K19}. The next result describes that every $\vert \cdot \vert$-prserving map is a function of bounded variation.

\begin{corollary}
Let $\mathbb{X}$ and $\mathbb{Y}$ be absolutely ordered spaces, and $f:\mathbb{X} \to \mathbb{Y}$ be an $\vert \cdot \vert$-preserving map. Then $f:[x,y] \to \mathbb{Y}$ is of bounded variation with $\mathcal{V}(f)= f(y)-f(x)$ for any $x,y \in \mathbb{X}$ with $x < y.$
\end{corollary}

\begin{proof}
Assume that $f:\mathbb{X} \to \mathbb{Y}$ be an $\vert \cdot \vert$-preserving map. Let $z,w \in \mathbb{X}$ with $z < w.$ Then $f(w)-f(z)=f(w-z)=f(\vert w-z \vert)=\vert f(w-z) \vert \geq 0$ so that $f(w) \geq f(z).$ Thus $f:[x,y] \to \mathbb{Y}$ is a monotonically increasing for any $x,y \in \mathbb{X}$ with $x < y.$ By Proposition \ref{1}, we get that $f\in \mathcal{BV}[x,y]$ with $\mathcal{V}(f)=\vert f(y)-f(x)\vert = f(y)-f(x).$
\end{proof}

Now, we prove one of the main theorem of this paper which elaborates that every function of bounded variation remains function of bounded variation on sub-intervals.

\begin{theorem}\label{5}
Let $\mathbb{Y}$ be a vector lattice which is dedekind complete and $x \leq z \leq y.$ Then $f\in \mathcal{BV}[x,y]$ if and only if $f\in \mathcal{BV}[x,z]$ and $f\in \mathcal{BV}[z,y].$ In this case, $\mathcal{V}(f,x,y)=\mathcal{V}(f,x,z)+\mathcal{V}(f,z,y).$
\end{theorem}

\begin{proof}
First assume that $f:[x,y] \to \mathbb{Y}$ is a function of bounded variation. Let $\mathcal{P}_1$ and $\mathcal{P}_2$ be partitions of $[x,z]$ and $[z,y]$ respectively. Then $\mathcal{P}=\mathcal{P}_1 \cup \mathcal{P}_2$ is partition of $[x,y].$ We have \begin{center} $\sum_{\mathcal{P}_1}[f] + \sum_{\mathcal{P}_2}[f] = \sum_\mathcal{P}[f] \leq V(f,x,y)$ \end{center} so that 
\begin{center}$\sum_{\mathcal{P}_1}[f] \leq V(f,x,y)$ and $\sum_{\mathcal{P}_2}[f]\leq V(f,x,y).$\end{center} Thus $f\in \mathcal{BV}[x,z]$ and $f\in \mathcal{BV}[z,y]$ with $\mathcal{V}(f,x,z) + \mathcal{V}(f,z,y) \leq \mathcal{V}(f,x,y).$

Conversely assume that $f\in \mathcal{BV}[x,z]$ and $f\in \mathcal{BV}[z,y].$ Let $\mathcal{P}$ be a partition of $[x,y].$ Put $\mathcal{P}^*=\mathcal{P} \cup \lbrace z \rbrace.$ Then $\mathcal{P}^*$ is also a partition of $[x,y]$ and there exist $\mathcal{P}_1$ and $\mathcal{P}_2$ partitions of $[x,z]$ and $[z,y]$ respectively such that $\mathcal{P}^*=\mathcal{P}_1 \cup \mathcal{P}_2.$ By Proposition \ref{10}, we have 

\begin{center}
$\sum_\mathcal{P}[f] \leq  \sum_{\mathcal{P}^*}[f] = \sum_{\mathcal{P}_1}[f] + \sum_{\mathcal{P}_2}[f]\leq \mathcal{V}(f,x,z)+\mathcal{V}(f,z,y).$
\end{center}

Thus $f\in \mathcal{BV}[x,y]$ with $\mathcal{V}(f,x,y) \leq \mathcal{V}(f,x,z)+\mathcal{V}(f,z,y).$ Hence, in this case, we get that $\mathcal{V}(f,x,y) = \mathcal{V}(f,x,z)+\mathcal{V}(f,z,y).$
\end{proof}

Next result characterize all the functions of bounded variation with zero total variation.

\begin{lemma}\label{4}
Let $f\in \mathcal{BV}[x,y].$ Then $f$ is constant if and only if $\mathcal{V}(f)=0.$
\end{lemma}

\begin{proof}
Assume that $f$ is constant. For any partition $\mathcal{P}$ of $[x,y],$ we get that $\displaystyle \sum_{i=1}^{n_{\mathcal{P}}}\vert f(x_i)-f(x_{i-1})\vert=0$  so that $\mathcal{V}(f)=0.$ Conversely assume that $\mathcal{V}(f)=0.$ By Proposition \ref{10}, we have $0\leq \vert f(z)-f(x)\vert \leq \mathcal{V}(f)=0$ for all $z\in [x,y].$ Then $\vert f(z)-f(x)\vert =0$ for all $z\in [x,y].$ In this case, $f(z)=f(x)$ for all $z\in [x,y].$ Thus $f$ is a constant function.
\end{proof}

\section{Norms on functions of bounded variations}

In this section, we show that the collection of functions of bounded variation forms ordered normed spaces.

Let $\mathbb{Y}$ be an absolutely ordered space and $f:[x,y]\to \mathbb{Y}$ be a function. For any partition $\mathcal{P}$ of $[x,y],$ we write: $\Sigma_\mathcal{P}^+[f]=\displaystyle \sum_{i=1}^{n_{\mathcal{P}}}[f(x_i)-f(x_{i-1})]^+$ and $\Sigma_\mathcal{P}^-[f]=\displaystyle \sum_{i=1}^{n_{\mathcal{P}}}[f(x_i)-f(x_{i-1})]^-.$ If $\sup \left \lbrace \sum_{\mathcal{P}}^+[f] :\textrm{$\mathcal{P}$ is a partition of $[x,y]$} \right \rbrace$ and $\sup \left \lbrace \sum_{\mathcal{P}}^-[f] :\textrm{$\mathcal{P}$ is a partition of $[x,y]$} \right \rbrace$ exist, we write \begin{center} $\mathcal{V}^+(f)=\sup \left \lbrace \sum_{\mathcal{P}}^+[f] :\textrm{$\mathcal{P}$ is a partition of $[x,y]$} \right \rbrace$ \end{center} and \begin{center} $\mathcal{V}^-(f)=\sup \left \lbrace \sum_{\mathcal{P}}^-[f] :\textrm{$\mathcal{P}$ is a partition of $[x,y]$} \right \rbrace.$ \end{center}

\begin{proposition}\label{9}
Let $\mathbb{Y}$ be an absolutely ordered space and $f:[x,y]\to \mathbb{Y}$ be a function. Then
\begin{enumerate}
\item[(1)] For any partition $\mathcal{P}$ of $[x,y],$ we have: $\Sigma_{\mathcal{P}}[f]=\Sigma_\mathcal{P}^+[f] + \Sigma_\mathcal{P}^-[f]$ and $f(y)-f(x)=\Sigma_\mathcal{P}^+[f] - \Sigma_\mathcal{P}^-[f].$ 
\end{enumerate}
If $\mathbb{Y}$ is a vector lattice, then
\begin{enumerate}
\item[(2)] For any partitions $\mathcal{P}_1$ and $\mathcal{P}_1$ of $[x,y]$ such that $\mathcal{P}_1 \subset \mathcal{P}_2$ we have: $\sum_{\mathcal{P}_1}^\pm[f]\leq \sum_{\mathcal{P}_2}^\pm[f].$ 
\end{enumerate}
Moreover, if $\mathbb{Y}$ is vector lattice which is dedekind complete and $f\in \mathcal{BV}[x,y],$ then
\begin{enumerate}
\item[(3)] $\mathcal{V}(f)=\mathcal{V}^+(f) + \mathcal{V}^-(f).$
\end{enumerate}
\end{proposition}

\begin{proof}
\begin{enumerate}
\item[(1)] For any partion $\mathcal{P}$ of $[x,y],$ let $f(x_i)-f(x_{i-1})=[f(x_i)-f(x_{i-1})]^+ - [f(x_i)-f(x_{i-1})]^-$ be orthogonal decomposition of $f$ for the sub-interval $[x_{i-1},x_i]$ for each $i.$ Then $\vert f(x_i)-f(x_{i-1})\vert = [f(x_i)-f(x_{i-1})]^+ + [f(x_i)-f(x_{i-1})]^-$ and consequently, we have:
\begin{eqnarray*}
\Sigma_{\mathcal{P}}[f]^+ + \Sigma_{\mathcal{P}}[f]^- &=& \displaystyle \sum_{i=1}^{n_{\mathcal{P}}} [f(x_i)-f(x_{i-1})]^+ + \displaystyle \sum_{i=1}^{n_{\mathcal{P}}} [f(x_i)-f(x_{i-1})]^- \\ 
&=& \displaystyle \sum_{i=1}^{n_{\mathcal{P}}} \lbrace [f(x_i)-f(x_{i-1})]^+ + [f(x_i)-f(x_{i-1})]^- \rbrace \\
&=& \displaystyle \sum_{i=1}^{n_{\mathcal{P}}} \vert f(x_i)-f(x_{i-1})\vert \\
&=&  \Sigma_{\mathcal{P}}[f]
\end{eqnarray*}

and

\begin{eqnarray*}
\Sigma_{\mathcal{P}}[f]^+ - \Sigma_{\mathcal{P}}[f]^- &=& \displaystyle \sum_{i=1}^{n_{\mathcal{P}}} [f(x_i)-f(x_{i-1})]^+ - \displaystyle \sum_{i=1}^{n_{\mathcal{P}}} [f(x_i)-f(x_{i-1})]^- \\ 
&=& \displaystyle \sum_{i=1}^{n_{\mathcal{P}}} \lbrace [f(x_i)-f(x_{i-1})]^+ - [f(x_i)-f(x_{i-1})]^- \rbrace \\
&=& \displaystyle \sum_{i=1}^{n_{\mathcal{P}}} f(x_i)-f(x_{i-1}) \\
&=& f(y)-f(x).
\end{eqnarray*}
\end{enumerate}

\begin{enumerate}
\item[(2)] Let $\mathcal{P}_1=\lbrace x=x_0 < x_1 < x_2 < \cdots < x_{n-1} < x_n=y\rbrace.$ Without loss of generality, we assume that $\mathcal{P}_2$ contains exactly one more point than $P_1.$ In this case, we have $\mathcal{P}_2=\lbrace x=x_0 < x_1 < x_2 < \cdots < x_{i-1} < z < x_i < \cdots < x_{n-1} < x_n=y\rbrace.$ Then
\begin{eqnarray*}
(f(x_i)-f(x_{i-1}))^\pm &=& \frac{1}{2}(\vert f(x_i)-f(x_{i-1})\vert \pm (f(x_i)-f(x_{i-1}))) \\
&\leq & \frac{1}{2}(\vert f(z)-f(x_{i-1})\vert \pm (f(z)-f(x_{i-1}))) \\
&+& \frac{1}{2}(\vert f(x_i)-f(z)\vert \pm (f(x_i)-f(z))) \\
&=& (f(z)-f(x_{i-1}))^\pm + (f(x_i)-f(z))^\pm
\end{eqnarray*} 
so that $\sum_{\mathcal{P}_1}^\pm[f]\leq \sum_{\mathcal{P}_2}^\pm[f].$
\end{enumerate}

\begin{enumerate}
\item[(3)] By (1), we have $\Sigma_{\mathcal{P}}[f]=\Sigma_\mathcal{P}^+[f] + \Sigma_\mathcal{P}^-[f]$ for every partition $\mathcal{P}$ of $[x,y].$ Thus $\Sigma_{\mathcal{P}}[f] \leq \mathcal{V}^+(f) + \mathcal{V}^-(f)$ so that $\mathcal{V}(f) \leq \mathcal{V}^+(f) + \mathcal{V}^-(f).$ Let $\mathcal{P}_1$ and $\mathcal{P}_2$ be partitions of $[x,y].$ Put $\mathcal{P}=\mathcal{P}_1 \cup \mathcal{P}_2.$ By (1) and (2), we get that $\Sigma_{\mathcal{P}_1}^+[f]+\Sigma_{\mathcal{P}_2}^-[f]\leq \Sigma_\mathcal{P}^+[f]+\Sigma_\mathcal{P}^-[f] = \Sigma_\mathcal{P}[f].$ Then $\Sigma_{\mathcal{P}_1}^+[f]+\Sigma_{\mathcal{P}_2}^-[f]\leq \mathcal{V}(f)$ so that $\mathcal{V}^+(f) + \mathcal{V}^-(f) \leq \mathcal{V}(f).$ Hence $\mathcal{V}(f)=\mathcal{V}^+(f) + \mathcal{V}^-(f).$
\end{enumerate}
\end{proof}

Now, we define the notion of variation function corresponding to a function of bounded variation.

\begin{definition}
Let $\mathbb{Y}$ be a vector lattice which is dedekind complete and $f\in \mathcal{BV}.$ By Theorem $\ref{5},$ we define a function $\mathcal{V}_f:[x,y]\to \mathbb{Y}$ such that $\mathcal{V}_f(z)=\mathcal{V}(f,x,z)$ for each $z\in [x,y].$ We call this function the variation function of $f.$
\end{definition}

Let's study some properties of the variation function.

\begin{theorem}\label{13}
Let $\mathbb{Y}$ be a vector lattice which is dedekind complete and $f\in \mathcal{BV}.$ Then the following statements hold:
\begin{enumerate}
\item[(1)] $\mathcal{V}_f$ is monotonically increasing such that $\mathcal{V}_f(x)=0.$
\item[(2)] $\mathcal{V}_f(z) \geq \vert f(z)-f(x)\vert.$
\item[(3)] $\mathcal{V}_f = f$ for every monotonically increasing function $f$ such that $f(x)=0.$
\end{enumerate}
For $\mathcal{V}^\pm_f(z) = \frac{1}{2}(\mathcal{V}_f(z) \pm (f(z)-f(x))),$ we also have:
\begin{enumerate}
\item[(4)] $\mathcal{V}_f^\pm(z) \geq 0.$
\item[(5)] $\mathcal{V}_f^\pm$ are monotonically increasing.
\item[(6)] $\mathcal{V}_f = \mathcal{V}^+_f + \mathcal{V}^-_f.$
\item[(7)] $f = f(x) + \mathcal{V}^+_f - \mathcal{V}^-_f.$
\end{enumerate}

\end{theorem}

\begin{proof} It is routine to verify (6) and (7). Next, we prove other statements.
\begin{enumerate}
\item[(1)] By Theorem \ref{5}, we have $\mathcal{V}(f,x,z_2)=\mathcal{V}(f,x,z_1)+\mathcal{V}(f,z_1,z_2)$ for $z_2\geq z_1.$ Since $\mathcal{V}(f,z_1,z_2)\geq 0,$ we get that $\mathcal{V}_f(z_2)\geq \mathcal{V}_f(z_1)$ for $z_2\geq z_1.$ Thus $\mathcal{V}_f$ is monotonically increasing.
\item[(2)] By Theorem \ref{5}, we have $f\in \mathcal{BV}[x,z]$ for every $z \in [x,y].$ Now by Corollary \ref{8}, we get that $\mathcal{V}_f(z) \geq \vert f(z)-f(x)\vert.$
\item[(3)] Let $f$ be monotonically increasing and $f(x)=0.$ By Proposition \ref{1}, we get that $\mathcal{V}_f(z)=f(z)-f(x)=f(z)$ for all $z\in [x,y].$ Thus $\mathcal{V}_f=f.$
\item[(4)] For each $z\in [x,y],$ we have $\pm (f(z)-f(x)) \leq \vert f(z)-f(x)\vert \leq \mathcal{V}_f(z).$ Thus $0 \leq \mathcal{V}_f^\pm(z).$
\item[(5)] Let $z_1,z_2 \in [x,y]$ such that $z_2\geq z_1.$ Then $2[\mathcal{V}_f^\pm (z_2)-\mathcal{V}_f^\pm (z_1)]=\mathcal{V}_f(z_2)-\mathcal{V}_f(z_1)\pm [f(z_2)-f(z_1)]=\mathcal{V}(f,z_1,z_2)\pm [f(z_2)-f(z_1)]\geq 0.$ Thus $\mathcal{V}_f^\pm$ are monotonically increasing.
\end{enumerate}
\end{proof}

\begin{corollary}\label{14}
Let $\mathbb{Y}$ be a vector lattice which is dedekind complete and $f,g\in \mathcal{BV}.$ Then $\pm (\mathcal{V}_f-\mathcal{V}_g)\leq \vert \mathcal{V}_f-\mathcal{V}_g\vert \leq \mathcal{V}_{f\pm g}\leq \mathcal{V}_f+\mathcal{V}_g.$ In particular, we have:
\begin{enumerate} 
\item[(1)] $\pm (\mathcal{V}(f)-\mathcal{V}(g))\leq \vert \mathcal{V}(f)-\mathcal{V}(g)\vert \leq \mathcal{V}(f\pm g).$
\item[(2)] $\pm (\mathcal{V}_f-\mathcal{V}_g)\leq \vert \mathcal{V}_f-\mathcal{V}_g\vert\leq \mathcal{V}_{\mathcal{V}_f \pm \mathcal{V}_g}\leq \mathcal{V}_f+\mathcal{V}_g.$ 
\end{enumerate}
\end{corollary}

\begin{proof}
By Theorems \ref{6} and \ref{5}, we have $\mathcal{V}_f(z)=\mathcal{V}(f,x,z)\leq \mathcal{V}(f\pm g,x,z)+\mathcal{V}(\mp g,x,z)=\mathcal{V}_{f\pm g}(z)+\mathcal{V}_g(z)\leq \mathcal{V}_f(z)+\mathcal{V}_g(z)+\mathcal{V}_g(z) .$ Then $\mathcal{V}_f(z)-\mathcal{V}_g(z)\leq \mathcal{V}_{f\pm g}(z)\leq \mathcal{V}_f(z)+\mathcal{V}_g(z).$ Interchanging $f$ and $g,$ we also have $\mathcal{V}_g(z)-\mathcal{V}_f(z)\leq \mathcal{V}_{f\pm g}(z)\leq \mathcal{V}_f(z)+\mathcal{V}_g(z).$ Thus $\pm (\mathcal{V}_f(z)-\mathcal{V}_g(z))\leq \mathcal{V}_{f\pm g}(z)\leq \mathcal{V}_f(z)+\mathcal{V}_g(z).$ Then $\pm (\mathcal{V}_f(z)-\mathcal{V}_g(z))\leq \vert \mathcal{V}_f(z)-\mathcal{V}_g(z)\vert \leq \mathcal{V}_{f\pm g}(z)\leq \mathcal{V}_f(z)+\mathcal{V}_g(z)$ so that $\pm (\mathcal{V}_f-\mathcal{V}_g)\leq \vert \mathcal{V}_f-\mathcal{V}_g\vert \leq \mathcal{V}_{f\pm g}\leq \mathcal{V}_f+\mathcal{V}_g.$ Putting $z=y,$ we immediately get (1). By Theorem \ref{13}(1), $\mathcal{V}_f$ and $\mathcal{V}_g$ are monotonically increasing functions such that $\mathcal{V}_f(x)=0$ and $\mathcal{V}_g(x)=0.$ Using Theorem \ref{13}(3) and replacing $f$ and $g$ by $\mathcal{V}_f$ and $\mathcal{V}_g$ respectively, we also get (2).
\end{proof}

Next, we show that $\mathcal{BV}$ forms an ordered normed space under the order structure of $\mathbb{Y}.$ 

\begin{theorem}\label{7}
Let $\mathbb{Y}$ be a vector lattice which is dedekind complete and $\mathcal{BV}^+=\lbrace f \in \mathcal{BV}:f([x,y])\subseteq \mathbb{Y}^+\rbrace.$ Put $f_0=e.$ Then 
\begin{enumerate}
\item[(1)] $(\mathcal{BV},\mathcal{BV}^+)$ forms an ordered space. 
\item[(2)] $(\mathcal{BV},\mathcal{BV}^+,f_0)$ forms an order unit space. 
\end{enumerate}
Moreover, for each $z\in [x,y]$ and $f\in \mathcal{BV},$ we write $\vert f\vert(z)=\vert f(z)\vert$ and define $\vert \cdot \vert:\mathcal{BV} \to \mathcal{BV}^+$ given by $f\mapsto \vert f\vert.$ Then
\begin{enumerate}
\item[(3)] $(\mathcal{BV},\mathcal{BV}^+,\vert \cdot \vert)$ forms an absolutely ordered space.
\end{enumerate}
Moreover, if $\mathbb{Y}$ also forms an $AM$-space, then 
\begin{enumerate}
\item[(4)] $(\mathcal{BV},\mathcal{BV}^+,f_0,\vert \cdot \vert)$ also forms an $AM$-space. In fact, in this case $\perp = \perp_\infty^a$ holds on $\mathcal{BV}^+$ so that it becomes an absolute order unit space.
\end{enumerate}
\end{theorem}

\begin{proof}
\begin{enumerate}
\item[(1)] By Theorem \ref{6}(2) and (3), we get that $\mathcal{BV}$ is a vector space. Next, let $f,g \in \mathcal{BV}^+.$ Then $(f+g)[x,y]=f([x,y])+g([x,y]) \subseteq \mathbb{Y}^+$ so that $f+g \in \mathcal{BV}^+.$ Thus $(\mathcal{BV},\mathcal{BV}^+)$ forms an ordered space. 
\item[(2)] let $f\in \mathcal{BV}.$ By Theorem \ref{6}(1), there exists $w\in \mathbb{Y}^+$ such that $\pm f(z)\leq \vert f(z)\vert \leq w.$ Since $w\leq \Vert w \Vert e,$ we get that $\Vert w\Vert e\pm f(z)\geq 0$ for all $z\in [x,y].$ Then $(\Vert w\Vert f_0 \pm f)([x,y]) \subseteq \mathbb{Y}^+.$ Thus $f_0=e$ is order unit for $(\mathcal{BV},\mathcal{BV}^+).$ 

Next, assume that $\pm f \in \mathcal{BV}^+.$ Then $\pm f(z)\in \mathbb{Y}^+$ for all $z\in [x,y].$ As $\mathbb{Y}^+$ is proper, we get that $f(z)=0$ for every $z\in [x,y].$ Thus $f=0$ so that $\mathcal{BV}^+$ is proper.

Finally, assume that $\epsilon f_0+f \in \mathcal{BV}^+$ for all $\epsilon > 0.$ For each fixed $z\in [x,y],$ we have $\epsilon e+f(z) \in \mathbb{Y}^+$ for all $\epsilon > 0.$ Since $\mathbb{Y}^+$ is Archimedean, we get that $f(z)\in \mathbb{Y}^+$ for every $z\in [x,y].$ Thus $f \in \mathcal{BV}^+$ so that $\mathcal{BV}^+$ is Archimedean. Hence $(\mathcal{BV},\mathcal{BV}^+,f_0)$ forms an order unit space. 

Since $\mathcal{BV}^+$ is proper and Archimedean, we get that the order unit $f_0$ defines a norm on $\mathcal{BV}$ in the following way: 
\begin{eqnarray*}
\Vert f\Vert_0 &=& \inf \lbrace \epsilon > 0:\epsilon f_0\pm f\in \mathcal{BV}^+ \rbrace \\
&=& \inf \lbrace \epsilon > 0:\epsilon e\pm f\in \mathcal{BV}^+ \rbrace \\
&=& \inf \lbrace \epsilon > 0:\epsilon e\pm f(z)\in \mathbb{Y}^+~\textrm{for all $z\in [x,y]$}\rbrace \\
&=& \inf \lbrace \epsilon > 0:\epsilon \geq \Vert f(z)\Vert ~\textrm{for all $z\in [x,y]$}\rbrace \\
&=& \inf \lbrace \epsilon > 0:\epsilon \geq \displaystyle \max_{z\in [x,y]} \Vert f(z)\Vert \rbrace \\
&=& \inf \lbrace \epsilon > 0:\epsilon \geq \displaystyle \Vert f\Vert_\infty \rbrace \\
&=& \Vert f\Vert_\infty.
\end{eqnarray*}

\item[(3)] By Theorem \ref{6}(4), the map $f\mapsto \vert f\vert$ is well defined from $\mathcal{BV}$ to $\mathcal{BV}^+.$ Let $f,g,h\in \mathcal{BV}$ and $z\in [x,y].$ Then
\begin{enumerate}
\item[(a)] Let $f\in \mathcal{BV}^+.$ Then $f(z)\in \mathbb{Y}^+$ so that $\vert f(z)\vert = f(z).$ Thus $\vert f\vert = f.$ 
\item[(b)] $\vert f\vert(z) \pm f(z)=\vert f(z)\vert\pm f(z)\in \mathbb{Y}^+,$ we get that $\vert f\vert \pm f \in \mathcal{BV}^+.$ 
\item[(c)] $\vert \alpha f\vert (z)=\vert \alpha f(z)\vert=\vert \alpha \vert \vert f(z)\vert=\vert \alpha \vert \vert f\vert(z),$ we get that $\vert \alpha f\vert=\vert \alpha\vert \vert f\vert.$
\item[(d)] Let $f,g,h\in \mathcal{BV}^+$ such that $\vert f-g\vert=f+g$ and $g-h \in \mathcal{BV}^+.$ Then $\vert f(z)-g(z)\vert=f(z)+g(z)$ and $0\leq h(z)\leq g(z).$ Thus $\vert f(z)-h(z)\vert=f(z)+h(z)$ so that $\vert f-h\vert=f+h.$
\item[(e)] Let $f,g,h\in \mathcal{BV}^+$ such that $\vert f-g\vert=f+g$ and $\vert f-h\vert=f+h.$ Then $\vert f(z)-g(z)\vert=f(z)+g(z)$ and $\vert f(z)-h(z)\vert=f(z)+h(z).$ Thus $\vert f(z)-\vert g(z)\pm h(z)\vert \vert=f(z)+\vert g(z)\pm h(z)\vert$ so that $\vert f-\vert g\pm h\vert\vert = f+\vert g\pm h\vert.$
\end{enumerate}

Thus $(\mathcal{BV},\mathcal{BV}^+,\vert \cdot \vert)$ forms an absolutely ordered space.

\item[(4)] Finally assume that $\mathbb{Y}$ is an $AM$-space. Let $f$ and $g\in \mathcal{BV}.$ Observe that $\Vert \vert f\vert \Vert_0=\Vert f\Vert_0.$ Without loss of generality, assume that $f$ and $g \in \mathcal{BV}^+.$ If $f\leq g,$ then $f(z)\leq g(z)\leq \Vert g\Vert_\infty e$ so that $\Vert f\Vert_0 \leq \Vert g\Vert_0.$ Next, we also have the following:
\begin{eqnarray*}
\Vert f \dot{\vee} g\Vert_0 &=& \Vert f \dot{\vee} g\Vert_\infty \\
&=& \displaystyle \max_{z\in [x,y]} \Vert (f \dot{\vee} g)(z)\Vert \\
&=& \displaystyle \max_{z\in [x,y]} \Vert f(z) \dot{\vee} g(z)\Vert \\
&\leq & \displaystyle \max_{z\in [x,y]} \lbrace \max \left \lbrace \Vert f(z)\Vert, \Vert g(z)\Vert \rbrace \right \rbrace \\
&=& \max \lbrace \Vert f\Vert_\infty, \Vert g\Vert_\infty \rbrace \\
&=& \max \lbrace \Vert f\Vert_0, \Vert g\Vert_0 \rbrace.
\end{eqnarray*}
Thus $\mathcal{BV}$ also forms $AM$-space. By Kakutani Theorem (see \cite{Kak}), we conclude that $\mathcal{BV}\cong C(K,\mathbb{R})$ for some compact Hausdorff space $K.$ Hence $\mathcal{BV}$ forms an absolute order unit space.
\end{enumerate}
\end{proof}

In the next result, we induce a new order structure on $\mathcal{BV}.$ 

\begin{theorem}\label{12}
Let $\mathbb{Y}$ be a vector lattice which is dedekind complete. Put $\mathcal{BV}_0^+=\lbrace f\in \mathcal{BV}^+: f~\textrm{is monotonically increasing} \rbrace.$ Then 
\begin{enumerate}
\item[(1)] $(\mathcal{BV},\mathcal{BV}_0^+)$ forms an ordered space. 
\end{enumerate}
For $f$ and $g\in \mathcal{BV},$ by $f\leq_0 g,$ we mean that $g-f\in \mathcal{BV}_0^+$ and we also define $\vert \cdot \vert_{\mathcal{V}}:\mathcal{BV} \to \mathcal{BV}_0^+$ given by $f\mapsto \vert f(x)\vert + \mathcal{V}_f.$ Then
\begin{enumerate}
\item[(2)] $\vert f\vert_\mathcal{V}=f$ for every $f\in \mathcal{BV}_0^+.$ 
\item[(3)] $\vert f\vert_\mathcal{V}\pm f \in \mathcal{BV}_0^+$ for every $f\in \mathcal{BV}.$
\item[(4)] $\vert \alpha f\vert_\mathcal{V}=\vert \alpha\vert \vert f\vert_\mathcal{V}$ for every $f\in \mathcal{BV}$ and $\alpha \in \mathbb{R}.$
\item[(5)] For $f,g$ and $h\in \mathcal{BV}_0^+$ such that $\vert f-g\vert_\mathcal{V}=f+g$ and $h-g \in \mathcal{BV}_0^+,$ we have $\vert f-h\vert_\mathcal{V}=f+h.$
\item[(6)] $\vert f\pm g\vert_\mathcal{V}\leq \vert f\vert_\mathcal{V}+\vert g\vert_\mathcal{V}$ for every pair $f$ and $g\in \mathcal{BV}.$
\end{enumerate}
For every pair $f$ and $g\in \mathcal{BV},$ we write: $f \dot{\vee} g:=\frac{1}{2}(f+g+\vert f-g\vert_\mathcal{V})$ and $f \dot{\wedge} g:=-\lbrace (-f)\dot{\vee}(-g)\rbrace=\frac{1}{2}(f+g-\vert f-g\vert_\mathcal{V}).$ Then 
\begin{enumerate}
\item[(7)] $f \dot{\wedge} g\leq_0 f,g\leq_0 f \dot{\vee} g.$ 
\end{enumerate}
\end{theorem}

\begin{proof}
\begin{enumerate}
\item[(1)] Let $f,g \in \mathcal{BV}_0^+.$ Then $f+g$ is also monotonically increasing and $f+g\in \mathcal{BV}^+$ so that $f+g \in \mathcal{BV}_0^+.$ Thus $(\mathcal{BV},\mathcal{BV}_0^+)$ forms an ordered space. 
\end{enumerate}

By Theorem \ref{13}(1), the map $f\mapsto \vert f(x)\vert + \mathcal{V}_f$ is well defined from $\mathcal{BV}$ to $\mathcal{BV}_0^+.$

\begin{enumerate}
\item[(2)] Let $f\in \mathcal{BV}_0^+.$ Then $f$ is monotonically increasing and $\vert f(x)\vert=f(x).$ By Proposition \ref{1}, we get that $\mathcal{V}_f(z)=\mathcal{V}(f,x,z)=f(z)-f(x).$ Thus $\vert f \vert_{\mathcal{V}}=\vert f(x)\vert +\mathcal{V}_f=f.$ 
\item[(3)] By Theorem \ref{13}(2) and (4), we get that 
\begin{eqnarray*}
\vert f\vert_\mathcal{V}(z)\pm f(z) &=& \vert f(x)\vert + \mathcal{V}_f(z) \pm f(z) \\
&=& \mathcal{V}_f(z) \pm f(z) \mp f(x)+\vert f(x)\vert \pm f(x) \\
&=& (\mathcal{V}_f(z) \pm (f(z)-f(x))) + (\vert f(x)\vert \pm f(x))\\ 
&=& 2\mathcal{V}_f^\pm(z)+(\vert f(x)\vert \pm f(x))\\
&\geq & 0
\end{eqnarray*}
for every $f\in \mathcal{BV}$ and $z\in [x,y].$ By Theorem \ref{13}(5), we conclude that $\vert f\vert_{\mathcal{V}}\pm f \in \mathcal{BV}_0^+$ for all $f \in \mathcal{BV}.$
\item[(4)] By Theorem \ref{6}(2), we have
\begin{eqnarray*}
\vert \alpha f\vert_\mathcal{V} &=& \vert \alpha f(x)\vert + \mathcal{V}(\alpha f) \\
&=& \vert \alpha \vert \vert f(x)\vert +\vert \alpha\vert \mathcal{V}(f) \\
&=& \vert \alpha\vert (\vert f(x)\vert + \mathcal{V}(f)) \\
&=& \vert \alpha \vert \vert f\vert_{\mathcal{V}}
\end{eqnarray*}
for every $f\in \mathcal{BV}$ and $\alpha \in \mathbb{R}.$
\end{enumerate}

Next, let $f,g$ and $h\in \mathcal{BV}_0^+.$

\begin{enumerate}
\item[(5)] Assume that $g-h \in \mathcal{BV}_0^+$ and $\vert f-g\vert_{\mathcal{V}}=f+g.$ Then $\vert f(x)-g(x)\vert + \mathcal{V}_{f-g}=f+g.$ Since $\mathcal{V}_f(x)=0,$ we have $\vert f(x)-g(x)\vert=f(x)+g(x).$ By Proposition \ref{1}, we get that $\mathcal{V}_{f-g}=\mathcal{V}_f+\mathcal{V}_g.$  As $g-h\in \mathcal{BV}_0^+,$ again by Proposition \ref{1}, we also get that $\mathcal{V}_{g-h}=(g-h)-(g(x)-h(x))=\mathcal{V}_g-\mathcal{V}_h.$ Now, by Corollary \ref{14}, it turns out that $\mathcal{V}_f + \mathcal{V}_h \geq \mathcal{V}_{f-h} \geq  \mathcal{V}_{f-g}-\mathcal{V}_{g-h} = \mathcal{V}_f+\mathcal{V}_h.$ Thus $\mathcal{V}_{f-h}=\mathcal{V}_f+\mathcal{V}_h.$ For $0\leq h(x)\leq g(x),$ we also have $\vert f(x)-h(x)\vert=f(x)+h(x).$ Finally, we conclude that $\vert f-h\vert_{\mathcal{V}}=f+h.$
\item[(6)] Again by Corollary \ref{14}, we get that
\begin{eqnarray*}
\vert f\pm g\vert_\mathcal{V} &=& \vert f(x)\pm g(x)\vert + \mathcal{V}_{f\pm g}\\
&\leq & \vert f(x)\vert + \vert g(x)\vert +\mathcal{V}_f+\mathcal{V}_g\\
&=& \vert f\vert_\mathcal{V}+\vert g\vert_\mathcal{V}.
\end{eqnarray*}
\item[(7)] By (3), we have $f\dot{\vee}g-f=\dfrac{1}{2}(\vert f-g\vert_\mathcal{V}-(f-g))\in \mathcal{BV}_0^+$ and $f-f\dot{\wedge}g=\dfrac{1}{2}(\vert f-g\vert_\mathcal{V}+(f-g))\in \mathcal{BV}_0^+$ so that $f\dot{\wedge} g\leq_0 f\leq_0 f \dot{\vee} g.$ Since $f\dot{\vee}g=g\dot{\vee}f$ and $f\dot{\wedge}g = g\dot{\wedge}f,$ we also get that $f\dot{\wedge}g\leq_0 g\leq_0 f \dot{\vee} g.$
\end{enumerate}
\end{proof}

The following result states that under the new ordering $\mathcal{BV}$ also forms an ordered normed space.

\begin{corollary}\label{15}
Given $f\in \mathcal{BV},$ we write: $\Vert f\Vert=\displaystyle \inf_{g\in \mathcal{BV}_0^+} \lbrace \Vert g\Vert_\infty : g\pm f\in \mathcal{BV}_0^+\rbrace.$ Then $(\mathcal{BV},\Vert \cdot \Vert)$ is an ordered normed space.
\end{corollary}

Finally, we induce another norm on $\mathcal{BV}$ by new ordering. Under certain condition, this norm turns out to be a complete norm. 

\begin{theorem}\label{3}
Let $\mathbb{Y}$ be an absolute order unit space having vector lattice structure which is dedekind complete. For each $f\in \mathcal{BV},$ we write: $\Vert f\Vert_{\mathcal{BV}}=\Vert \vert f(x)\vert+\mathcal{V}(f)\Vert.$ Then $(\mathcal{BV},\Vert \cdot \Vert_{\mathcal{BV}})$ forms a normed space. Moreover, we have
\begin{enumerate}
\item[(1)] $\Vert f \Vert_\mathcal{BV}= \Vert \vert f(x)\vert + \mathcal{V}_f\Vert_0.$ 
\item[(2)] $\Vert f\Vert_0 \leq \Vert f\Vert_\mathcal{BV}.$
\item[(3)] If $\mathbb{Y}$ is order complete, then $(\mathcal{BV},\Vert \cdot \Vert_{\mathcal{BV}})$ is complete.
\item[(4)] If $\vert f\vert_\mathcal{V} \leq \vert g\vert_\mathcal{V},$ then $\Vert f\Vert_\mathcal{BV}\leq \Vert g\Vert_\mathcal{BV}.$
\item[(5)] $\Vert f \dot{\vee} g\Vert_0 \leq \Vert f\Vert_\mathcal{BV} + \Vert g\Vert_\mathcal{BV}$ for all $f,g\in \mathcal{BV}^+.$
\end{enumerate}
In this case, $(\mathcal{BV},\Vert \cdot \Vert_{\mathcal{BV}})$ forms an ordered normed space.
\end{theorem}

\begin{proof}
Let $f,g\in \mathcal{BV}$ and $\alpha \in \mathbb{R}.$ By Lemma \ref{4}, we have $f=0$ if and only if $f(x)=0$ and $\mathcal{V}(f)=0$ if and only if $\Vert f\Vert_{\mathcal{BV}}=0.$ Next, by Theorems \ref{2} and \ref{6}(3), we get that $\vert f(x)+g(x)\vert+\mathcal{V}(f+g)\leq \vert f(x)\vert + \vert g(x)\vert+\mathcal{V}(f)+\mathcal{V}(g).$ Then $\Vert f+g\Vert_{\mathcal{BV}} \leq \Vert \vert f(x)\vert + \vert g(x)\vert+\mathcal{V}(f)+\mathcal{V}(g)\Vert \leq \Vert \vert f(x)\vert+\mathcal{V}(f)\Vert_{\mathcal{BV}} + \Vert \vert g(x)\vert+\mathcal{V}(g)\Vert = \Vert f\Vert_{\mathcal{BV}} + \Vert g\Vert_{\mathcal{BV}}.$ Finally, by Theorem \ref{6}(2), we conclude that $\vert \alpha f\vert+\mathcal{V}(\alpha f)=\vert \alpha \vert \vert f(x)\vert +\vert \alpha \vert \mathcal{V}(f)=\vert \alpha \vert (\vert f(x)\vert + \mathcal{V}(f))$ so that $\Vert \alpha f \Vert_{\mathcal{BV}} = \vert \alpha\vert \Vert f\Vert_{\mathcal{BV}}.$ Thus $(\mathcal{BV},\Vert \cdot \Vert_{\mathcal{BV}})$ forms a normed space. Now, we prove other properties.
\begin{enumerate}
\item[(1)] By Theorem \ref{13}(1), $\mathcal{V}_f$ is monotonically increasing function. For $z_1,z_2\in [x,y]$ such that $z_1\leq z_2,$ we have $0\leq \vert f(x)\vert+\mathcal{V}_f(z_1)\leq \vert f(x)\vert+\mathcal{V}_f(z_2)\leq \vert f(x)\vert+\mathcal{V}_f(y)=\vert f(x)\vert+\mathcal{V}(f)$ so that $\Vert \vert f(x)\vert + \mathcal{V}_f(z_1)\Vert\leq\Vert \vert f(x)\vert + \mathcal{V}_f(z_2)\Vert \leq \Vert \vert f(x)\vert + \mathcal{V}_f(y)\Vert=\Vert \vert f(x)\vert + \mathcal{V}(f)\Vert.$ Thus $\Vert f \Vert_\mathcal{BV}=\Vert \vert f(x)\vert + \mathcal{V}_f\Vert_0.$
\item[(2)] By Theorem \ref{13}(2), we have $\vert f(z)\vert \leq \vert f(z)-f(x)\vert +\vert f(x)\vert\leq \mathcal{V}_f(z)+\vert f(x)\vert.$ By (1), we get that $\Vert f(z)\Vert \leq \Vert \mathcal{V}_f(z)+\vert f(x)\vert\Vert \leq \Vert f\Vert_\mathcal{BV}.$ Thus $\Vert f\Vert_0 \leq \Vert f\Vert_\mathcal{BV}.$
\item[(3)] Let $\lbrace f_n \rbrace$ be a cauchey sequence in $\mathcal{BV}.$ Then $\vert f_n(x)-f_m(x)\vert \to 0$ and $\mathcal{V}(f_n-f_m)\to 0.$ In this case, $\vert f_n(z)-f_m(z)\vert \to 0$ for every $z\in [x,y].$ Define $f(z)=\displaystyle \lim_{n\to \infty} f_n(z).$ For any partition $\mathcal{P}$ of $[x,y],$ we have $\displaystyle \sum_{i=1}^{n_\mathcal{P}}\vert f_n(x_i)-f_n(x_{i-1})\vert \leq \mathcal{V}(f_n-f_m)+\mathcal{V}(f_m).$ Since $\lbrace f_n \rbrace$ is a cauchey sequence, given any $c\in \mathbb{Y}^+,$ we can find $m_0\in \mathbb{N}$ such that $\mathcal{V}(f_n-f_m)\leq c$ and $\mathcal{V}(f_m)\leq c$ for all $n,m\geq m_0.$ In this case, we have $\displaystyle \sum_{i=1}^{n_\mathcal{P}}\vert f_n(x_i)-f_n(x_{i-1})\vert \leq 2 c$ for all $n\geq m_0.$ Letting $n\to \infty,$ we get that $\displaystyle \sum_{i=1}^{n_\mathcal{P}}\vert f(x_i)-f(x_{i-1})\vert \leq 2c$ so that $f\in \mathcal{BV}.$ Similarily, letting $n\to \infty$ and keeping $m$ fixed in $\displaystyle \sum_{i=1}^{n_\mathcal{P}}\vert (f_n(x_i)-f_m(x_i))-(f_n(x_{i-1})-f_m(x_{i-1}))\vert \leq \mathcal{V}(f_n-f_m)\leq c,$ we conclude that $\displaystyle \sum_{i=1}^{n_\mathcal{P}}\vert (f(x_i)-f_m(x_i))-(f(x_{i-1})-f_m(x_{i-1}))\vert\leq c$ for all $m\geq m_0.$ Thus $f_n \to f$ so that $(\mathcal{BV},\Vert \cdot \Vert_{\mathcal{BV}})$ forms a complete space.
\item[(4)] It is trivial to verify.
\item[(5)] By Corollary \ref{14}, for $f,g\in \mathcal{BV}^+$ and $z\in [x,y],$ we have
\begin{eqnarray*}
\Vert (f\dot{\vee}g)(z)\Vert &=& \frac{1}{2}\Vert f(z)+g(z)+\vert f(x)-g(x)\vert +\mathcal{V}_{f-g}(z)\Vert \\
&\leq & \frac{1}{2}\Vert f(z)+g(z)+\vert f(x)\vert + \vert g(x)\vert +\mathcal{V}_f(z)+\mathcal{V}_g(z)\Vert \\
&=& \frac{1}{2}\Vert f(z)+g(z)+\vert f(x)\vert + \vert g(x)\vert +\mathcal{V}_f(z)+\mathcal{V}_g(z)\Vert \\
&\leq & \frac{1}{2}(\Vert f(z)\Vert +\Vert g(z)\Vert +\Vert \vert f(x)\vert +\mathcal{V}_f(z)\Vert + \Vert \vert g(x)\vert +\mathcal{V}_g(z)\Vert) \\
&\leq & \frac{1}{2}(\Vert f\Vert_0 +\Vert g\Vert_0 +\Vert f\Vert_\mathcal{BV} + \Vert g\Vert_\mathcal{BV})\\
&\leq & \frac{1}{2}(\Vert f\Vert_\mathcal{BV} +\Vert g\Vert_\mathcal{BV} +\Vert f\Vert_\mathcal{BV} + \Vert g\Vert_\mathcal{BV})\\
&=& \Vert f\Vert_\mathcal{BV}+\Vert g\Vert_\mathcal{BV}
\end{eqnarray*}

so that $\Vert f \dot{\vee} g\Vert_0 \leq \Vert f\Vert_\mathcal{BV} + \Vert g\Vert_\mathcal{BV}.$
\end{enumerate}
\end{proof}

\begin{corollary}\label{16}
Given $f\in \mathcal{BV},$ we write: $\Vert f\Vert=\displaystyle \inf_{g\in \mathcal{BV}_0^+} \lbrace \Vert g\Vert_\mathcal{BV} : g\pm f\in \mathcal{BV}_0^+\rbrace.$ Then $(\mathcal{BV},\Vert \cdot \Vert)$ is an ordered normed space.
\end{corollary}


\begin{thebibliography}{100}
	
\bibitem{MA71} E. M. Alfsen, Compact Convex sets and Boundary Integrals, Springer-Verlag, Berlin-Heidelberg-New York, 1971. 

\bibitem {CDOB} Charalambos D. Aliprantis and Owen Burkinshaw, Positive operators, Springer, 2006.

\bibitem{B06} B. Blackadar, Operator algebras. Theory of C$^*$-algebras and von Neumann algebras, Springer-Verlag, Berlin- Heidelberg-New York, 2006.	

\bibitem {LSSF} Lamberto Cesari, "Sulle funzioni a variazione limitata", Annali della Scuola Normale Superiore, {\bf 5} (1936), 299–313.

\bibitem {LCLD} Lamberto Cesari, L'opera di Leonida Tonelli e la sua influenza nel pensiero scientifico del secolo, 1986.

\bibitem {JBCF} John B. Conway, Functions of one complex variable, Springer-Verlag, Berlin-Heidelberg-New York, 1973. 

\bibitem{GN} I. Gelfand and M. Neumark, On the imbedding of normed rings into the ring of operators in Hilbert space, Rec. Math. [Mat. Sbornik] N.S., {\bf 54} (1943), 197--213. 

\bibitem {J72} G. Jameson, Ordered linear spaces, Lecture Notes in mathematics, Springer-Verlag, Berlin-Heidelberg-New York, {\bf 141} (1970).

\bibitem {CJFS} Camille Jordan, "Sur la série de Fourier" [On Fourier's series], Comptes rendus hebdomadaires des séances de l'Académie des sciences, {\bf 92} (1881), 228--230. 

\bibitem{RVK} R.\ V.\ Kadison, A representation theory for commutative topological algebra, Mem. Amer. Math. Soc., {\bf 7} (1951). 

\bibitem{Kad51} R. V. Kadison, Order properties of bounded self-adjoint operators, Proc. Amer. Math. Soc., {\bf 2} (1951), 505--510. 

\bibitem{RJ83} R. V. Kadison and J. R. Ringrose, Fundamentals of the theory of operator algebras, Academic Press, Inc., London-New York,  1983. 

\bibitem{Kak} S. Kakutani, Concrete representation of abstract $M$-spaces, Ann. of Math., {\bf 42} (1941), 994--1024.

\bibitem{K10} A. K. Karn, A p-theory of ordered normed spaces, Positivity, {\bf 14} (2010), 441--458. 

\bibitem{K14} A. K. Karn, Orthogonality in $l_p$-spaces and its bearing on ordered Banach spaces, Positivity, {\bf 18} (2014), 223--234. 

\bibitem{K16} A. K. Karn, Orthogonality in C$^*$-algebras, Positivity, {\bf 20} (2016), 607--620.

\bibitem{K18} A. K. Karn, Algebraic orthogonality and commuting projections in operator algebras, Acta Sci. Math. (Szeged), {\bf 84} (2018), 323--353. 

\bibitem{K19} A. K. Karn and A. Kumar, Isometries of absolute order unit spaces, Positivity, {\bf 24} (2020), 1263--1277.

\bibitem{GKP} G. K. Pedersen, C$^*$-algebras and their automorphism groups, Academic Press, Inc., London-New York, 1979.

\bibitem{HHS74} H. H. Schaefer, Banach lattices and positive operators, Springer-Verlag, Berlin-Heidelberg-New York, 1974.
	
\bibitem {RLWA} Richard L. Wheeden and Antoni Zygmund, Measure and Integral: An Introduction to Real Analysis, CRC Pure and Applied Mathematics, 1977.	
	
\bibitem {WN73} Y. C. Wong and K. F. Ng, Partially ordered topological vector spaces, Oxford Mathematical Monographs, Clarendon Press, Oxford, 1973.	
 
\end{thebibliography}
\end{document}